\documentclass[a4paper,11pt]{amsart}
\usepackage{graphicx}
\usepackage{amsmath}
\usepackage{amssymb}
\usepackage{oldgerm}
\usepackage[all]{xy}
\usepackage{tikz}
\usepackage{tikz-cd}
\usepackage{comment}
\usepackage{mathtools}
\usepackage{textcomp}
\usepackage{hyperref}
\usepackage[shortlabels]{enumitem}

\newtheorem{Theorem}{Theorem}[section]
\newtheorem{Proposition}[Theorem]{Proposition}
\newtheorem{lemm}[Theorem]{Lemma}
\newtheorem{Corollary}[Theorem]{Corollary}
\theoremstyle{definition}

\newtheorem{Remark}[Theorem]{Remark}

\newcommand{\iso}{\cong}
\newcommand{\an}{a_0 \otimes \cdots \otimes a_n}

\newcommand{\Hom}{\mathrm{Hom}}
\newcommand{\bcirc}{\,\bar{\circ}\,}

\newcommand{\bs}{\text{\textbackslash}}

\title[Batalin-Vilkovisky structure]{Batalin-Vilkovisky structure on Hochschild cohomology with coefficients in the dual algebra}

\author{Marco Armenta}
\author{Samuel Leblanc}

\date{}

\AtEndDocument{\bigskip{\small%
  \textsc{Institut quantique, Universit\'e de Sherbrooke\\ 
  Sherbrooke, QC J1K 2R1, Canada} \par  
  \textit{E-mail address}: 
  \texttt{Marco.Armenta@USherbrooke.ca} \par
  \addvspace{\medskipamount}
  \textsc{D\'epartement de math\'ematiques, Universit\'e de Sherbrooke\\ Sherbrooke, QC J1K 2R1, Canada} \par  
  \textit{E-mail address}: 
  \texttt{Samuel.Leblanc6@USherbrooke.ca}
}}

\begin{document}
\maketitle

\baselineskip=15pt

\begin{abstract}
We prove that Hochschild cohomology with coefficients in $A^*=\Hom_k(A,k)$ under conditions on the algebra structure of $A^*$ is a Batalin-Vilkovisky algebra. 
We also show that for symmetric and Frobenius algebras, this recovers the known BV-structures in Hochschild cohomology with coefficients in $A$ but admits an easy-to-describe BV-operator. 
Finally, we show that for monomial algebras $A = kQ/\langle T \rangle$, the Hochschild cohomology with coefficients in $A^*$ is always a Batalin-Vilkovisky algebra. 
\end{abstract}

\section{Introduction}

Let $A$ be an associative unital algebra projective over a commutative associative unital ring $k$. 
The Hochschild cohomology $k$-modules of $A$ with coefficients in an $A$-bimodule $M$,
\[
H^\bullet(A,M)=\bigoplus_{n\geq0} H^n(A,M)
\]
have been introduced by Hochschild \cite{Hochschild} and extensively studied since then. 
Operations on cohomology have been defined, such as the cup product and the Gerstenhaber bracket, making it into a Gerstenhaber algebra \cite{Gerstenhaber}. 
Tradler showed \cite{Tradler} that for symmetric algebras this Gerstenhaber algebra structure on cohomology comes from a Batalin-Vilkovisky operator (BV-operator) and Menichi extended the result \cite{Menichi}. 
As Tradler mentions, it is important to determine other families of algebras where this property holds. 
Lambre-Zhou-Zimmermann proved that this is the case for Frobenius algebras with semisimple Nakayama automorphism \cite{Lambre}. 
Independently, Volkov proved with other methods that this holds for Frobenius algebras in which the Nakayama automorphism has finite order and the characteristic of the field $k$ does not divide it \cite{Volkov}. 
It has also been shown that Calabi-Yau algebras admit the existence of a BV-operator \cite{Ginzburg}, \cite{Kowalzig}, and that this BV-structure on its cohomology is isomorphic to the one of the cohomology of the Koszul dual, for a Koszul Calabi-Yau algebra \cite{Chen}. 
More generally, for algebras with duality, see \cite{Lambre}, a BV-structure is equivalent to a Tamarkin-Tsygan calculus or a differential calculus \cite{Lambre}. 
The proofs of \cite{Ginzburg}, \cite{Lambre} and \cite{Tradler} have in common the use of Connes' differential \cite{Connes} on homology to define the BV-operator on cohomology. 
More recently, efforts have been put into explicitly describing the BV-structure on the Hochschild cohomology ring of different classes of algebras, for instance group algebras \cite{LiuZhou}, exterior algebras \cite{Lu}, skew Calabi-Yau generalized Weyl algebras \cite{LiuMa}, self-injective quadratic monomial algebras \cite{GaoHou}, and self-injective Nakayama algebras \cite{Itagaki}.

We start by giving an interpretation of Connes' differential in Hochschild cohomology with coefficients in the $A$-bimodule $A^*=\Hom_k(A,k)$. 
When $A^*$ is also an associative $k$-algebra for which the multiplication verifies a few conditions, we can define a graded commutative cup product on $H^\bullet(A,A^*)$. 
In this case, we show in Theorem \ref{thm:BVdual} that we automatically get a BV-algebra structure on $H^\bullet(A,A^*)$ with an easy-to-describe BV-operator. 
Given a morphism of $A$-bimodules $\psi : A^* \to A$ again satisfying a few conditions, we can define a bracket analogous to the usual Gerstenhaber bracket \cite{Gerstenhaber}, and we extend the aforementioned  theorem by proving, in Theorem \ref{thm:BVDualPsi}, that $H^\bullet(A,A^*)$ is a BV-algebra with this bracket. 
Afterwards, we show that this applies to symmetric, and more generally, Frobenius algebras (Corollary \ref{cor:symm} and Theorem \ref{thm:Frobenius}). 
More precisely, the BV-algebras obtained are isomorphic to the known BV-algebras with coefficients in $A$ (see \cite{Lambre}, \cite{Volkov}). 
In those cases, we are able to give an explicit formulation of the BV-operator as opposed to the ones on $HH^\bullet(A)$. 
But, by doing so, we lose some simplicity in the description of the cup product. 
As a direct corollary, one sees that the BV-operators defined in \cite{Lambre} and \cite{Volkov} are actually equal. 
Finally, we show that Theorem \ref{thm:BVdual} applies when $A = kQ/\langle T\rangle$ is a monomial path algebra, which implies that $H^\bullet(A,A^*)$ is always a BV-algebra. 
In general, it's not the case that $HH^\bullet(A)$ admits a BV-structure if one considers the usual Gerstenhaber bracket. 
Furthermore, as opposed to with symmetric and Frobenius algebras, we are able to give an explicit description of the cup product. 

\section{Connes' differential}\label{section:connes}

\textit{Connes' differential} (\cite{Connes}) is the map $B:HH_n(A) \to HH_{n+1}(A)$ that makes the Hochschild theory of an algebra into a differential calculus \cite{Tamarkin}. In terms of the bar resolution, it is given by
\[
  	  B([\an]) = \left[ \sum_{i=0}^n (-1)^{ni} 1 \otimes a_i \otimes \cdots \otimes a_n \otimes a_0 \otimes \cdots \otimes a_{i-1} \right].
\]
For an $A$-bimodule $M$ the \textit{dual} $A$-bimodule is denoted $M^*=\Hom_k(M,k)$. We consider the canonical $A$-bimodule structure on $M^*$, that is $(afb)(x)=f(bxa)$ for all $a,b\in A$, all $f\in M^*$ and all $x\in M$. Let 
\[
\bar{B}: H^{n+1}(A,A^*) \to H^{n}(A,A^*)
\]
be given by 
\[
    \bar{B}([f])(a_1 \otimes \cdots \otimes a_n)(a_0) := \sum_{i=0}^{n} (-1)^{ni} f(a_i \otimes \cdots \otimes a_n \otimes a_0 \otimes \cdots \otimes a_{i-1}) (1). 
\]
It is straightforward to verify that it is well-defined. 
Let 
\[
\mathfrak{C}: H^n(A,M^*) \to  H_n(A,M)^*
\]
be the morphism
\[
    \mathfrak{C}([f])([x \otimes a_1 \otimes \cdots \otimes a_n])=f(a_1 \otimes \cdots \otimes a_n)(x),
\]
for all $a_i \in A$, for $i=1,\cdots,n$, all $x\in M$ and all $[f]\in H^{n}(A,M^*)$, see \cite{Cartan} (p. 181, Exercise 8). The evaluation map $ev:H_n(A,M) \to H_n(A,M)^{**}$ can be composed with the $k$-dual of $\mathfrak{C}$ to get a morphism 
\[
\varphi: H_n(A,M) \to H^n(A,M^*)^*
\]
which is given by 
\[
	\varphi([x \otimes a_1 \otimes \cdots \otimes a_n])([f]) = f(a_1 \otimes \cdots \otimes a_n)(x).
\]
For $M=A$ we obtain a morphism $\varphi: HH_n(A) \to H^n(A,A^*)^*$. 
%The proof of the following lemma is straightforward.
% proposition5
\begin{lemm}\label{lemm:varphi}
    Let $k$ be a commutative associative unital ring and let $A$ be an associative and unital $k$-algebra. The following diagram is commutative
    \[
	\begin{tikzpicture}
         	 \matrix (m) [matrix of math nodes,row sep=2em,column sep=2em]
          	{
           	  	HH_n(A) & HH_{n+1}(A) \\
           	  	H^n(A,A^*)^* & H^{n+1}(A,A^*)^*.\\
           	 };
          	\path[-stealth]
           	 (m-1-1) edge node [above] {$B$} (m-1-2)
           	 (m-2-1) edge node [above] {$\bar{B}^*$} (m-2-2)
            
           	 (m-1-1) edge node [left] {$\varphi$} (m-2-1)
           	 (m-1-2) edge node [right] {$\varphi$} (m-2-2);        
	\end{tikzpicture}
	\]
	If $k$ is a field then $\varphi$ is a monomorphism. If $k$ is a field and $HH_n(A)$ is finite-dimensional then $\varphi: HH_n(A) \to H^n(A,A^*)^*$ is an isomorphism.
\end{lemm}

\begin{proof}
Let $[\an] \in HH_n(A)$. For $[f] \in H^{n+1}(A,A^*)$ we have
\begin{align*}
	\varphi(B([&a_0 \otimes \cdots \otimes a_n])) ([f]) \\
	% = \varphi \left( \left[  \displaystyle\sum_{i=0}^n (-1)^{in} 1 \otimes a_i \otimes \cdots \otimes a_n \otimes a_0 \otimes \cdots \otimes a_{i-1} \right] \right) ([f]) \\
	 &=  \displaystyle\sum_{i=0}^n (-1)^{in} \varphi ( 1 \otimes a_i \otimes \cdots \otimes a_n \otimes a_0 \otimes \cdots \otimes a_{i-1} ) ([f]) \\
	 &=  \displaystyle\sum_{i=0}^n (-1)^{in} f(a_i \otimes \cdots \otimes a_n \otimes a_0 \otimes \cdots \otimes a_{i-1} )(1) \\
	 &=  \bar{B}([f])(a_1 \otimes \cdots \otimes a_n)(a_0) \\
	 &=   \varphi([\an])(\bar{B}([f]))\\
	 &=  \bar{B}^* ( \varphi([a_0 \otimes \cdots a_n]) )([f]).
\end{align*}
If $k$ is a field, then the evaluation map is a monomorphism and $\mathfrak{C}$ is an isomorphism \cite{Cartan}, hence $\varphi$ is a monomorphism. If in addition $HH_n(A)$ is finite dimensional over the field $k$, the evaluation map on $HH_n(A)$ is an isomorphism and then so is $\varphi$.
\end{proof}

\section{Batalin-Vilkovisky structure}

For the rest of this paper, we work over a field $k$. 
A \textit{Gerstenhaber algebra} is a triple $\left( \mathcal{H}^\bullet, \cup, [\ ,\ ] \right)$ such that $\mathcal{H}^\bullet$ is a graded $k$-module, $\cup:\mathcal{H}^n \times \mathcal{H}^m \to \mathcal{H}^{n+m}$ is a graded commutative associative product and $[\ ,\ ]:\mathcal{H}^n \times \mathcal{H}^m \to \mathcal{H}^{n+m-1}$ is a graded Lie bracket such that it is anti-symmetric $[f,g] = -(-1)^{(|f|-1)(|g|-1)} [g,f]$, it satisfies the Jacobi identity
    \[
    [f,[g,h]] = [[f,g],h] + (-1)^{(|f|-1)(|g|-1)} [g,[f,h]]
    \]
as well as the Poisson identity
    \[
     [f,g \cup h] = [f,g] \cup h + (-1)^{(|f|-1)|g|} g \cup [f,h],
    \]
for all homogeneous elements $f,g,h$ of $\mathcal{H}^\bullet$. We denote by $|f|$ the degree of a homogeneous element $f\in\mathcal{H}^\bullet$.
A \textit{Batalin-Vilkovisky algebra} (\textit{BV-algebra}) is a Gerstenhaber algebra $(\mathcal{H}^*,\cup,[\ ,\ ])$ together with a morphism 
\[
\Delta:\mathcal{H}^{\bullet+1} \to \mathcal{H}^\bullet
\]
such that $\Delta^2 = 0$ and
\begin{equation}\label{eq:BVOp}
[f,g] = (-1)^{|f|} \big( \Delta(f \cup g) - \Delta(f) \cup g - (-1)^{|f|} f \cup \Delta(g) \big).
\end{equation}

When $M = A$, the Hochschild cohomology group $HH^\bullet(A)$ admits a graded commutative cup product sending $f$ and $g$ to the map given by 
\[
f\cup g(a_{1} \otimes \cdots \otimes a_{n+m}) = f(a_{1}\otimes \cdots \otimes a_n)g(a_{n+1}\otimes \cdots \otimes a_m).
\] 
More generally, if $M$ is an algebra, we can define a cup product on $H^\bullet(A,M)$ analogously. 
Under the hypotheses stated in the following proposition, this cup product is also graded commutative. 

\begin{Proposition}\label{prp:multiplicationCup}
If $M$ is an $A$-bimodule that is also an associative $k$-algebra for which
\begin{itemize}
    \item $a(mm') = (am)m'$,
    \item $(mm')a = m(m'a)$,
    \item $(ma)m' = m(am')$
\end{itemize}
hold for all $m,m' \in M$ and $a \in A$, then the cup product on $H^\bullet(A,M)$ is graded commutative.
\end{Proposition}

\begin{proof}
This is obtained from similar computations as in \cite{Gerstenhaber}, see Remark 1.3.5 in \cite{Witherspoon}.
\end{proof}

In the case $M=A^*$, we automatically get a BV-algebra. 

\begin{Theorem}\label{thm:BVdual}
    Let $k$ be a field and let $A$ be a $k$-algebra. 
    If the $A$-bimodule $A^*$ admits an associative $k$-algebra structure on which the cup product on $H^\bullet(A,A^*)$ is graded commutative, then $H^\bullet(A,A^*)$ is a BV-algebra with BV-operator $\bar{B}$. 
\end{Theorem}

\begin{proof}
    Since the following diagram is commutative
     \[
	\begin{tikzpicture}
         	 \matrix (m) [matrix of math nodes,row sep=2em,column sep=2em]
          	{
           	  	HH_n(A) & HH_{n+1}(A) & HH_{n+2}(A)\\
           	  	H^n(A,A^*)^* & H^{n+1}(A,A^*)^* & H^{n+2}(A,A^*)^*\\
           	 };
          	\path[-stealth]
           	 (m-1-1) edge node [above] {$B$} (m-1-2)
           	 (m-1-2) edge node [above] {$B$} (m-1-3)
           	 
           	 (m-2-1) edge node [above] {$\bar{B}^*$} (m-2-2)
           	 (m-2-2) edge node [above] {$\bar{B}^*$} (m-2-3)
           	 
           	 (m-1-1) edge node [left] {$\varphi$} (m-2-1)
           	 (m-1-2) edge node [left] {$\varphi$} (m-2-2)
           	 (m-1-3) edge node [right] {$\varphi$} (m-2-3);        
	\end{tikzpicture}
	\]
    and by Lemma \ref{lemm:varphi}, $\varphi$ is a monomorphism, we have that $\bar{B}^2=0$. 
    Write $\cup_*$ the cup product on $H^\bullet(A,A^*)$ and define
    \[
    [f,g]_* := (-1)^{|f|} \big( \bar{B}(f \cup_* g) - \bar{B}(f) \cup_* g - (-1)^{|f|} f \cup_* \bar{B}(g)  \big).
    \]
    Since $\cup_*$ is graded commutative, the data $(H^\bullet(A,A^*), \cup_*, [\, , \,]_*, \bar{B})$ is a BV-algebra.
\end{proof}

Usually, on $HH^\bullet(A)$, one defines the bracket by
\[
[f, g] := f \bcirc g - (-1)^{(n-1)(m-1)}g\bcirc f 
\]
where, for $[f] \in HH^n(A)$ and $[g] \in HH^m(A)$, 
\begin{align*}
&f \bcirc g(a_1 \otimes \cdots \otimes a_{n+m-1}) := \\
&\sum_{i=1}^{n}(-1)^{t_i}f(a_1 \otimes \cdots \otimes a_{i-1} \otimes g(a_i \otimes \cdots \otimes a_{i+m-1}) \otimes a_{i+m} \otimes \cdots \otimes a_{n+m-1})
\end{align*}
with $t_i = (i-1)(m-1)$. 
Then, the relation given in equation \ref{eq:BVOp} becomes something to prove, not a definition as was done in the proof of the previous theorem. 
However, since we are interested in the Hochschild cohomology of $A$ with coefficients in $A^*$, we can't use the above definition of $[\,,\,]$. 
Suppose that the $A$-bimodule $A^*$ is an associative $k$-algebra. 
Following Section 9 of \cite{Gerstenhaber}, given a morphism of $A$-bimodules $\psi : A^* \to A$ such that 
\begin{equation}\label{eq:conditionPsi}
\psi(u)v = uv = u\psi(v)
\end{equation}
for all $u,v \in A^*$, we define 
\begin{equation}\label{eq:bracketPsi}
[f, g]_\psi := f \bcirc (\psi\circ g) - (-1)^{(n-1)(m-1)}g\bcirc (\psi \circ f). 
\end{equation}

\begin{Theorem}\label{thm:BVDualPsi}
    Let $k$ be a field and let $A$ be a $k$-algebra. 
    If the $A$-bimodule $A^*$ is an associative $k$-algebra where the multiplication respects equation \ref{eq:conditionPsi} for a morphism of $A$-bimodules $\psi : A^* \to A$, then $H^\bullet(A,A^*)$ is a BV-algebra with bracket $[\,,\,]_\psi$ and BV-operator $\bar{B}$. 
\end{Theorem}

\begin{proof}
If the multiplication in $A^*$ satisfies equation \ref{eq:conditionPsi}, then the conditions of Proposition \ref{prp:multiplicationCup} are respected, which implies that the cup product is graded commutative. 
The rest of the proof is a direct adaptation of the proof of Theorem 1 in \cite{Tradler} with $\Delta = \bar{B}$.
\end{proof}

\section{Frobenius and Symmetric algebras}
Assume that $A$ is a \textit{symmetric algebra}, i.e. a finite-dimensional $k$-algebra with a symmetric, associative, and non-degenerate bilinear form $\langle -, -\rangle: A\times A \to k$, where associative means
\[
\langle ab,c \rangle=\langle a,bc \rangle
\]
for all $a,b,c\in A$. 
The bilinear form defines an isomorphism of $A$-bimodules $Z:A \to A^*$ given by $Z(a)=\langle a,- \rangle$. 
It is shown in \cite{Tradler}, Theorem 1, that this defines a BV-operator on Hochschild cohomology, where $\Delta f$ is defined such that for $[f] \in HH^n(A)$ we have
    \[
        \langle\Delta f (a_1 \otimes \cdots \otimes a_{n-1}),a_n\rangle = \sum_{i=1}^n (-1)^{i(n-1)}\langle f(a_i \otimes \cdots a_n \otimes a_1 \cdots \otimes a_{i-1}),1\rangle .
    \]

Furthermore, the isomorphism $Z$ endows $A^*$ with an associative algebra structure, with multiplication 
\begin{equation}\label{eq:multSymmetric}
f\cdot g = Z(Z^{-1}(f)Z^{-1}(g)) = Z^{-1}(f)g = fZ^{-1}(g).
\end{equation}
The last two equalities are because
\[
Z(ab) = \langle ab, -\rangle = \langle a, b-\rangle = \langle a, - \rangle b = Z(a)b
\]
and 
\[
Z(ab) = \langle ab, -\rangle = \langle -, ab\rangle = \langle -a, b \rangle = \langle b , -a\rangle = a\langle b , -\rangle = aZ(b).
\]
Since $Z^{-1}$ is a morphism of $A$-bimodules, we can apply Theorem \ref{thm:BVDualPsi} to obtain that the cup product $\cup_*$ is graded commutative, that we can define the bracket $[\,,\,]_* := [\,,\,]_{Z^{-1}}$, and that $H^\bullet(A,A^*)$ is a BV-algebra. 
Moreover:

\begin{Corollary}\label{cor:symm}
    If $A$ is a symmetric algebra, then the BV-algebras $HH^\bullet(A)$ and $H^\bullet(A,A^*)$ are isomorphic.
\end{Corollary}

\begin{proof}
    Let $Z:A \to A^*$ be the isomorphism of $A$-bimodules given by the bilinear form of $A$. 
    We will denote $Z_*:HH^\bullet(A) \to H^\bullet(A,A^*)$ the isomorphism induced by composition with $Z$. Then the following diagram is commutative
     \[
	\begin{tikzpicture}
         	 \matrix (m) [matrix of math nodes,row sep=2em,column sep=2em]
          	{
           	  	HH^n(A) & HH^{n-1}(A)\\
           	  	H^n(A,A^*) & H^{n-1}(A,A^*).\\
           	 };
          	\path[-stealth]
           	 (m-1-1) edge node [above] {$\Delta$} (m-1-2)
           	 (m-2-1) edge node [above] {$\bar{B}$} (m-2-2)
            
           	 (m-1-1) edge node [left] {$Z_*$} (m-2-1)
           	 (m-1-2) edge node [right] {$Z_*$} (m-2-2);        
	\end{tikzpicture}
	\]
	Indeed, 
	\[
	\begin{array}{l}
	     (\bar{B} \circ Z_*)([f]) (a_1 \otimes \cdots \otimes a_{n-1})(a_0) \\ 
	      =  \bar{B}(Z_* \circ f)(a_1 \otimes \cdots \otimes a_{n-1})(a_0)  \\
	      =  \sum_{i=0}^{n-1} (-1)^{(n-1)i} (Z_*\circ f)(a_i \otimes \cdots \otimes a_{n-1} \otimes a_0 \otimes \cdots \otimes a_{i-1}) (1) \\
        = \sum_{i=0}^{n-1} (-1)^{(n-1)i} \langle f(a_i\otimes \cdots \otimes a_{n-1}\otimes a_0 \otimes \cdots \otimes a_{i-1}), 1\rangle\\
        = \langle \Delta f(a_1 \otimes \cdots \otimes a_{n-1}), a_0\rangle\\
	      =  (Z_* \circ \Delta)([f]) (a_1 \otimes \cdots \otimes a_{n-1})(a_0).
	\end{array}
	\]
 Even more, there are commutative diagrams where the vertical maps are isomorphisms
	\[
	\begin{tikzpicture}
         	 \matrix (m) [matrix of math nodes,row sep=2em,column sep=2em]
          	{
           	  	HH^n(A) \times HH^m(A) & HH^{n+m}(A)\\
           	  		H^n(A,A^*) \times H^m(A,A^*) & H^{n+m}(A,A^*)\\
           	 };
          	\path[-stealth]
           	 (m-1-1) edge node [above] {$\cup$} (m-1-2)
           	 (m-2-1) edge node [above] {$\cup_*$} (m-2-2)
            
           	 (m-1-1) edge node [left] {$Z_* \times Z_*$} (m-2-1)
           	 (m-1-2) edge node [right] {$Z_*$} (m-2-2);        
	\end{tikzpicture}
	\]
and
	\[
	\begin{tikzpicture}
         	 \matrix (m) [matrix of math nodes,row sep=2em,column sep=2em]
          	{
           	  	HH^n(A) \times HH^m(A) & HH^{n+m-1}(A)\\
           	  		H^n(A,A^*) \times H^m(A,A^*) & H^{n+m-1}(A,A^*).\\
           	 };
          	\path[-stealth]
           	 (m-1-1) edge node [above] {$[\ ,\ ]$} (m-1-2)
           	 (m-2-1) edge node [above] {$[\ ,\ ]_*$} (m-2-2)
            
           	 (m-1-1) edge node [left] {$Z_* \times Z_*$} (m-2-1)
           	 (m-1-2) edge node [right] {$Z_*$} (m-2-2);        
	\end{tikzpicture}
	\]
	Indeed,
\begin{align*}
	Z_*(f&) \cup_* Z_*(g)(a_1\otimes \cdots \otimes a_{n+m})\\
	 &= Z\Big(Z^{-1}\big(Z(f(a_1\otimes \cdots \otimes a_n))\big)Z^{-1}\big(Z(g(a_{n+1}\otimes \cdots \otimes a_{n+m}))\big)\Big)\\
	 &= Z(f(a_1\otimes \cdots \otimes a_n)g(a_{n+1}\otimes \cdots \otimes a_{n+m}))\\
	 &= Z_*(f\cup g)(a_1\otimes \cdots \otimes a_{n+m})
\end{align*}
	and
\begin{align*}
	    [Z_*f,&Z_*g]_* \\
	    &= (-1)^{|f|} \big( \bar{B}\big( Z_*f \cup_* Z_*g \big) - \bar{B}(Z_*f)\cup_* Z_*g - (-1)^{|f|}Z_*f \cup_* \bar{B}(Z_*g) \big)  \\
	    &= (-1)^{|f|} \big( \bar{B}\big( Z_* (f \cup g) \big) - Z_*(\Delta f) \cup_* Z_*g - (-1)^{|f|} Z_*f \cup_* Z_* (\Delta g)\big) \\
	    &= (-1)^{|f|} \big( Z_* \Delta(f \cup g)  - Z_*(\Delta f \cup g) - (-1)^{|f|} Z_*(f \cup \Delta g) \big) \\
	    &= (-1)^{|f|} Z_* \big( \Delta(f \cup g)  - (\Delta f \cup g) - (-1)^{|f|} (f \cup \Delta g) \big) \\
	    &= Z_* [f,g].
\end{align*}
	Commutativity of these diagrams implies that the $BV$-algebras $HH^\bullet(A)$ and $H^\bullet(A,A^*)$ are isomorphic.
\end{proof}

\begin{Remark}
Observe that choosing $\Delta:=(Z_*)^{-1}  \bar{B}  Z_* $ gives $HH^\bullet(A)$ the structure of a BV-algebra.
\end{Remark} 

Assume now that $A$ is a \textit{Frobenius algebra}, i.e. a finite dimensional $k$-algebra with a non-degenerate associative bilinear form $\langle -,-\rangle:A \times A \to k$. For every $a\in A$ there exists a unique $\mathfrak{N}(a) \in A$ such that $\langle a,-\rangle = \langle -, \mathfrak{N}(a)\rangle$. The map $\mathfrak{N}:A \to A$ turns out to be an algebra isomorphism and is called the \textit{Nakayama automorphism} of the Frobenius algebra $A$. 
Therefore, a symmetric algebra is a Frobenius algebra with Nakayama automorphism $\mathfrak{N} = \mathrm{id}$. 
The next result is a generalization of the previous corollary. 
Following \cite{Lambre} we consider the $A$-bimodule $A_\mathfrak{N}$ whose underlying $k$-module is $A$ and the corresponding actions are
\[
    a x b = a x \mathfrak{N} (b).
\]
Hence, the morphism $\xi: A_\mathfrak{N} \to A^*$ given by $\xi(a) = \langle -, a\rangle$ is an isomorphism of $A$-bimodules. 
By composing the evaluation map $A \to A^{**}$ with the transpose of $\xi$, we obtain an isomorphism of $A$-bimodules $Z: A \to A_\mathfrak{N}^*$ given by $Z(a) = \langle a, - \rangle$. 
Here, $A_\mathfrak{N}^*$ is endowed with the canonical bimodule structure $(afb)(c) = f(bc\mathfrak{N}(a))$, for $a,b \in A$, $c \in A_\mathfrak{N}$, and $f \in A_\mathfrak{N}^*$. 
The bimodule $A_\mathfrak{N}^*$ is an associative $k$-algebra with multiplication given by the same formula as for symmetric algebras (equation \ref{eq:multSymmetric}). 
To now verify the equalities given in the equation, we again have
\[
Z(ab) = \langle ab, -\rangle = \langle a, b-\rangle = \langle a, - \rangle b = Z(a)b,
\]
but for the second one, we have
\[
Z(ab) \!=\! \langle ab, -\rangle \!=\! \langle -, \mathfrak{N}(ab)\rangle \!=\! \langle -\mathfrak{N}(a), \mathfrak{N}(b) \rangle \!=\! \langle b , -\mathfrak{N}(a)\rangle \!=\! a\langle b , -\rangle \!=\! aZ(b).
\]

Since $Z^{-1}$ is a morphism of $A$-bimodules and by equation \ref{eq:multSymmetric}, we get a graded commutative cup product $\cup_*$ and we can define the bracket $[\,,\,]_* := [\,,\,]_{Z^{-1}}$. 
Therefore, $H^\bullet(A, A_\mathfrak{N}^*)$ is a Gerstenhaber algebra. 
However, we can't directly apply Theorem \ref{thm:BVDualPsi}, since the morphism $Z^{-1}$ has domain $A_\mathfrak{N}^*$, not $A^*$. 
With hypotheses on $\mathfrak{N}$, we can find a BV-operator on $H^\bullet(A, A_\mathfrak{N}^*)$ that turns it into a BV-algebra. 
Under the hypothesis of the following theorem, this BV-algebra is isomorphic to the construction from \cite{Lambre}. 

\begin{Theorem}\label{thm:Frobenius}
Let $A$ be a Frobenius algebra with semisimple Nakayama automorphism $\mathfrak{N}$. Then, the BV-algebras $HH^\bullet(A)$ and $H^\bullet(A,A_\mathfrak{N}^*)$ are isomorphic.
\end{Theorem}

\begin{proof}
It has been shown in \cite{Lambre}, Theorem 4.1, that when $\mathfrak{N}$ is semisimple, i.e. diagonalizable over the algebraic closure of $k$, then $HH^\bullet(A)$ is a BV-algebra. 
The BV-operator is $\partial B_\mathfrak{N}^* \partial^{-1}$, where $B_\mathfrak{N} : H_{\bullet}(A, A_\mathfrak{N}) \to H_{\bullet+1}(A, A_\mathfrak{N})$ sends $[a_0\otimes a_1\otimes \cdots \otimes a_n]$ to 
\[
\left[
\sum_{i=0}^n(-1)^{in}1\otimes a_i \otimes \cdots \otimes a_n \otimes a_0 \otimes \mathfrak{N}(a_1)\otimes \cdots \otimes \mathfrak{N}(a_{i-1})
\right],
\]
as defined in \cite{Kowalzig}, and where the map $\partial : H_\bullet(A, A_\mathfrak{N})^* \to HH^\bullet(A)$ is the isomorphism with inverse $\partial^{-1}([f]) = (-1)^{|f|}\langle-, f\rangle$. 
We consider the BV-operator on $H^\bullet(A,A_\mathfrak{N}^*)$ defined as $\Delta := Z_*\partial B_\mathfrak{N}^*\overline{\mathfrak{C}}$ with $Z_* : HH^\bullet(A) \to H^\bullet(A,A_\mathfrak{N}^*)$ the morphism induced by composition with $Z$ and where $\overline{\mathfrak{C}}([f]) := (-1)^{|f|}\mathfrak{C}([f])$ is the isomorphism defined in Section \ref{section:connes}. 
The situation is summarized in the diagram below:
\[
\begin{tikzcd}
H_n(A,A_\mathfrak{N})^* \arrow[d, "\partial", swap]\arrow[r, "B_\mathfrak{N}^*"] & H_{n-1}(A,A_\mathfrak{N})^* \arrow[d, "\partial"]\\
HH^n(A) \arrow[d, "Z_*", swap] \arrow[r, "\partial B_\mathfrak{N}^*\partial^{-1}", swap] & HH^{n-1}(A)\arrow[d, "Z_*"]\\
H^n(A, A_\mathfrak{N}^*) \arrow[uu, bend left=45, "\overline{\mathfrak{C}}"] \arrow[r, "\Delta", swap] & H^{n-1}(A, A_\mathfrak{N}^*).
\end{tikzcd}
\]
We shall now show that $B_\mathfrak{N}^*\overline{\mathfrak{C}}Z_* = B_\mathfrak{N}^*\partial^{-1}$. 
On one hand, we have
\begin{align*}
(&B_\mathfrak{N}^*\overline{\mathfrak{C}}Z_*)([f])([a_0\otimes a_1 \otimes \cdots \otimes a_{n-1}])\\
&= \overline{\mathfrak{C}}(Z_* f)(B_\mathfrak{N}(a_0\otimes a_1 \otimes \cdots \otimes a_{n-1}))\\
&= (-1)^n\!\sum_{i=0}^{n-1}\!(-1)^{i(n-1)}\!(Z_*f)(a_i \otimes \cdots \otimes a_n \otimes a_0 \otimes \mathfrak{N}(a_1)\otimes \cdots \otimes \mathfrak{N}(a_{i-1}))(1)\\
&= \sum_{i=0}^{n-1}(-1)^{i(n-1)n}\langle f(a_i \otimes \cdots \otimes a_n \otimes a_0 \otimes \mathfrak{N}(a_1)\otimes \cdots \otimes \mathfrak{N}(a_{i-1})), 1\rangle
\end{align*}
and on the other,
\begin{align*}
(B_\mathfrak{N}^*&\partial^{-1})([f])([a_0\otimes a_1 \otimes \cdots \otimes a_{n-1}])\\
&= (-1)^n\langle -, f \rangle(B_\mathfrak{N}(a_0\otimes a_1 \otimes \cdots \otimes a_{n-1}))\\
&= \sum_{i=1}^{n-1}(-1)^{i(n-1)n}\langle 1, f(a_i \otimes \cdots \otimes a_n \otimes a_0 \otimes \mathfrak{N}(a_1)\otimes \cdots \otimes \mathfrak{N}(a_{i-1}))\rangle\\
&= \sum_{i=1}^{n-1}(-1)^{i(n-1)n}\langle f(a_i \otimes \cdots \otimes a_n \otimes a_0 \otimes \mathfrak{N}(a_1)\otimes \cdots \otimes \mathfrak{N}(a_{i-1})), \mathfrak{N}(1)\rangle
\end{align*}
with $\mathfrak{N}(1) = 1$. 
We directly obtain that $\Delta Z_* = Z_* \partial B_\mathfrak{N}^* \partial^{-1}$ and that $\Delta$ is indeed a BV-operator. 
In particular, $\Delta^2 = 0$. 
Following similar computations as in the proof of the previous corollary, we can show that there is an isomorphism of BV-algebra $HH^\bullet(A) \iso H^\bullet(A,A_\mathfrak{N}^*)$. 
\end{proof}

Perhaps surprisingly, we are able to give an explicit formula for the BV-operator $\Delta$. 
This is, as far as we are aware, not the case for the BV-operator $\partial B_\mathfrak{N}^* \partial^{-1}$ on $HH^\bullet(A)$, since $\partial$ is hard to describe, as mentioned in Remark 3.1 of \cite{Lambre}. 
Define, analogous to $\bar{B}$, the map $\overline{B}_\mathfrak{N}: H^{n}(A,A_{\mathfrak{N}}^*) \to H^{n-1}(A,A_{\mathfrak{N}}^*)$ given by 
\begin{align*}
\overline{B}_{\mathfrak{N}}(&[f])(a_1 \otimes \cdots \otimes a_{n-1})(a_0) :=\\
& \sum_{i=0}^{n-1} (-1)^{i(n-1)} f(a_i \otimes \cdots \otimes a_{n-1} \otimes a_0 \otimes \mathfrak{N}(a_1)\otimes \cdots \otimes \mathfrak{N}(a_{i-1})) (1). 
\end{align*}
We obtain the following result. 

\begin{Corollary}
The BV-operator on $H^\bullet(A, A_\mathfrak{N}^*)$ is 
$\Delta = \overline{B}_\mathfrak{N}$. 
\end{Corollary}

\begin{proof}
First, we see that 
\begin{align*}
\Delta([&f])(a_1\otimes \cdots \otimes a_{n-1})(a_0)\\
&= Z_*\partial B_\mathfrak{N}^*\overline{\mathfrak{C}}([f])(a_1\otimes \cdots \otimes a_{n-1})(a_0)\\
&= \langle \partial B_\mathfrak{N}^*\overline{\mathfrak{C}}([f])(a_1\otimes \cdots \otimes a_{n-1}), a_0 \rangle\\
&= (-1)^n\langle a_0, \mathfrak{N}(\partial B_{\mathfrak{N}}^*\mathfrak{C}([f])(a_1\otimes \cdots \otimes a_{n-1}))\rangle.
\end{align*}
Since Su\'arez-\'Alvarez showed in \cite{SuarezAlvarez} that the Nakayama automorphism acts trivially on the Hochschild cohomology of $A$, we get the equality 
\[\mathfrak{N}(\partial B_{\mathfrak{N}}^*\mathfrak{C}([f])(a_1\otimes \cdots \otimes a_{n-1})) = \partial B_{\mathfrak{N}}^*\mathfrak{C}([f])(a_1\otimes \cdots \otimes a_{n-1}).
\]
It follows that
\begin{align*}
\Delta([&f])(a_1\otimes \cdots \otimes a_{n-1})(a_0)\\
&= (-1)^n\langle a_0, \partial B_{\mathfrak{N}}^*\mathfrak{C}([f])(a_1\otimes \cdots \otimes a_{n-1})\rangle\\
&= \partial^{-1}(\partial B_\mathfrak{N}^*\mathfrak{C}([f]))(a_0\otimes a_1 \otimes \cdots \otimes a_{n-1})\\
&= \mathfrak{C}([f])(B_\mathfrak{N}(a_0\otimes a_1 \otimes \cdots \otimes a_{n-1}))\\
&= \overline{B}_\mathfrak{N}([f])(a_1 \otimes \cdots \otimes a_{n-1})(a_0)
\end{align*}
which proves the corollary. 
\end{proof}

Furthermore, we automatically see that $\overline{B}_{\mathrm{id}} = \bar{B}$ and that Corollary \ref{cor:symm} indeed follows from the previous two results. 

Using a result of Volkov \cite{Volkov}, the hypothesis on $\mathfrak{N}$ in Theorem \ref{thm:Frobenius} can be changed. 

\begin{Corollary}
Let $A$ be a Frobenius algebra with Nakayama automorphism $\mathfrak{N}$. 
If $k$ is algebraically closed, the order of $\mathfrak{N}$ is finite, and the characteristic of $k$ does not divide the order of $\mathfrak{N}$, then the BV-algebras $HH^\bullet(A)$ and $H^\bullet(A, A_\mathfrak{N}^*)$ are isomorphic. 
\end{Corollary}

\begin{proof}
By Corollary 3 of \cite{Volkov}, $HH^\bullet(A)$ is a BV-algebra with BV-operator $\Delta : HH^\bullet(A) \to HH^{\bullet-1}(A)$ defined by the relation
\begin{align*}
\langle \Delta f&(a_1 \otimes \cdots \otimes a_{n-1}), a_0\rangle = \\ 
&\sum_{i=0}^{n-1}(-1)^{i(n-1)}\langle f(a_i \otimes \cdots \otimes a_{n-1} \otimes a_0 \otimes \mathfrak{N}(a_1)\otimes \cdots \otimes \mathfrak{N}(a_{i-1})), 1\rangle.
\end{align*}
Denoting again by $Z_*$ the isomorphism $HH^\bullet(A) \to H^\bullet(A, A_\mathfrak{N}^*)$ induced by composition with $Z$, we get that the following diagram is commutative
 \[
	\begin{tikzpicture}
         	 \matrix (m) [matrix of math nodes,row sep=2em,column sep=2em]
          	{
           	  	HH^n(A) & HH^{n-1}(A)\\
           	  	H^n(A,A^*) & H^{n-1}(A,A^*).\\
           	 };
          	\path[-stealth]
           	 (m-1-1) edge node [above] {$\Delta$} (m-1-2)
           	 (m-2-1) edge node [above] {$\overline{B}_\mathfrak{N}$} (m-2-2)
            
           	 (m-1-1) edge node [left] {$Z_*$} (m-2-1)
           	 (m-1-2) edge node [right] {$Z_*$} (m-2-2);        
	\end{tikzpicture}
	\]
	Indeed, similarly as in the proof of Corollary \ref{cor:symm},
\begin{align*}
 (&\overline{B}_\mathfrak{N} \circ Z_*)([f]) (a_1 \otimes \cdots \otimes a_{n-1})(a_0) \\ 
&= \sum_{i=0}^{n-1}(-1)^{i(n-1)}\langle f(a_i \otimes \cdots \otimes a_{n-1} \otimes a_0 \otimes \mathfrak{N}(a_1)\otimes \cdots \otimes \mathfrak{N}(a_{i-1})), 1\rangle\\
&= \langle \Delta f(a_1 \otimes \cdots \otimes a_{n-1}), a_0\rangle\\
  &=  (Z_* \circ \Delta)([f]) (a_1 \otimes \cdots \otimes a_{n-1})(a_0).
\end{align*}
Again, like in Corollary \ref{cor:symm}, we can show the isomorphism of the cup product and the bracket.
\end{proof}

The authors of \cite{Lambre} and \cite{Volkov} both used Connes' differential to define their BV-operators. 
But, at first glance, it's not obvious that they are equal.

\begin{Remark}
The BV-operator $\partial B_\mathfrak{N}^*\partial^{-1}$ of \cite{Lambre} and the BV-operator $\Delta$ of \cite{Volkov} are equal. 
Indeed, we just showed that
\[
\partial B_\mathfrak{N}^*\partial^{-1} = Z_*^{-1}\overline{B}_\mathfrak{N}Z_* = \Delta.
\]
\end{Remark}

\section{Monomial path algebras}

Let $Q$ be a finite quiver with $n$ vertices and consider a \textit{monomial path algebra} $A=kQ/\langle T\rangle$, that is, $T$ is a subset of paths in $Q$ of length greater than or equal to 2. 
We do not require the algebra $A$ to be finite dimensional. 
We write $s(\omega)$ and $t(\omega)$ for the source and the target of $\omega$ and we denote by $e_1, \ldots, e_n$ the idempotents of $A$ given by the vertices of $Q$. 
A basis $P$ of $A$ is given by the set of paths of $Q$ which do not contain paths of $T$. 
Let $P^\vee$ be the dual basis of $P$, and for $\omega \in P$ we denote $\omega^\vee$ its dual. 
Let $\alpha \in P$ and define $\omega_{/ \alpha}$ as the subpath of $\omega$ that starts in $s(\omega)$ and ends in $s(\alpha)$ if $\alpha$ is a subpath of $\omega$ such that $t(\alpha)=t(\omega)$, and zero otherwise. 
Let $\beta \in P$ and define $_{\beta \bs }\omega$ as the subpath of $\omega$ that starts at $t(\beta)$ and ends in $t(\omega)$ if $\beta$ is a subpath of $\omega$ such that $s(\beta)=s(\omega)$, and zero otherwise. 
In particular, $_{e_i \bs}e_{i/e_{i}} = e_i$, $\omega_{/\omega} = e_{s(\omega)}$, and $_{\omega\bs}\omega = e_{t(\omega)}$. 
The canonical $A$-bimodule structure of $A^*$ is isomorphic to the one given by linearly extending the following action
\[
         \alpha.\omega^\vee.\beta = (_{\beta \bs }\omega_{/ \alpha})^\vee.
\]
Now we construct a multiplication on $A^*$ so that it becomes an associative $k$-algebra satisfying the conditions of Proposition \ref{prp:multiplicationCup}. 
In general, we can't find a morphism $\psi$ that respects the conditions of equation \ref{eq:conditionPsi} to apply Theorem \ref{thm:BVDualPsi}. 
For $\omega, \gamma \in P$ we define
\[
    \omega^\vee \cdot \gamma^\vee = \left\{  \begin{array}{ll}
    (\gamma \omega)^\vee & \text{ if } \ t(\omega)=s(\gamma), \\
      0  & \text{ otherwise} \\
   \end{array} \right.
\]
and extend by linearity. Observe that $\gamma \ _{\beta \bs }\omega = \gamma_{/ \beta} \ \omega$, then
\[
    (\omega^\vee.\beta) \cdot \gamma^\vee = (_{\beta \bs }\omega)^\vee \cdot \gamma^\vee = (\gamma \ _{\beta \bs }\omega)^\vee = (\gamma_{/ \beta} \ \omega)^\vee = \omega^\vee \cdot (\gamma_{/ \beta})^\vee = \omega^\vee \cdot (\beta.\gamma^\vee).
\]
Also,
\[
    \alpha . (\omega^\vee \cdot \gamma^\vee) . \beta = \alpha.(\gamma \omega)^\vee.\beta = (_{\beta \bs }\gamma \omega_{/ \alpha})^\vee = (\omega_{/ \alpha})^\vee \cdot (_{\beta \bs }\gamma)^\vee = (\alpha.\omega^\vee)\cdot (\gamma^\vee.\beta).
\]
The multiplication is associative since the product of $A$, i.e. the composition of paths, is associative. 
Therefore, Proposition \ref{prp:multiplicationCup} applies and we get a direct corollary of Theorem \ref{thm:BVdual}. 

\begin{Corollary}
    Let $A$ be a monomial path algebra. 
    Then $H^\bullet(A,A^*)$ is a BV-algebra.
\end{Corollary}

In general, for a monomial path algebra $A = kQ/\langle T \rangle$, the Hochschild cohomology group $HH^\bullet(A)$ is not a BV-algebra, if one defines the bracket as in \cite{Gerstenhaber} and discussed above. 
For instance, if $Q$ is the quiver
\[
\begin{tikzcd}[sep=small]
    & & 3\arrow[dr, "\alpha_3"] \\
    & 2 \arrow[ur, "\alpha_2"] \arrow[rr, "\beta", swap] & & 4\arrow[dr, "\alpha_4"]\\
    1 \arrow[ur, "\alpha_1"] \arrow[rrrr, "\gamma", swap] & & & & 5
\end{tikzcd}
\]
and $T$ is formed by all paths of length 2, then the algebra $A$ is a radical square zero algebra and $HH^\bullet (A)$ is not a BV-algebra. 
Indeed, from Example 3.3 in \cite{Bustamante}, we know that the Gerstanhaber bracket $[\, , \, ]$ on $HH^\bullet (A)$ is not trivial, but the cup product is (in degree at least one), by Theorem 3.1 of \cite{Bustamante}. 
However, if $HH^\bullet(A)$ admitted a BV-operator, then by the relation of equation \ref{eq:BVOp}, the bracket would be trivial, which would give an obvious contradiction. 

\begin{Remark}
If $A$ is a Frobenius algebra, then it is possible that $HH^\bullet(A)$ admits a BV-structure, as seen earlier, but that could depend on the characteristic of the base field (see \cite{Lambre}, \cite{Volkov}), which is not the case for $H^\bullet(A,A^*)$. 
\end{Remark}

%\subsection{Gorenstein monomial algebras}

\subsection*{Acknowledgements}
The first author would like to thank Claude Cibils and Ricardo Campos for useful discussions at IMAG, Universit\'e de Montpellier, during the Rencontre 2018 du GdR de Topologie Alg\'ebrique. The first author received funds from CIMAT, CNRS, CONACyT and EDUCAFIN. 

\bibliographystyle{amsplain}
\bibliography{main}

\providecommand{\bysame}{\leavevmode\hbox to3em{\hrulefill}\thinspace}
\providecommand{\MR}{\relax\ifhmode\unskip\space\fi MR }
\providecommand{\MRhref}[2]{%
  \href{http://www.ams.org/mathscinet-getitem?mr=#1}{#2}
}
\providecommand{\href}[2]{#2}
\begin{thebibliography}{10}

\bibitem{Bustamante}
J.C. Bustamante, \emph{The cohomology structure of string algebras}, J. Pure.
  Appl. Algebra \textbf{204} (2006), no.~3, 616--626.

\bibitem{Cartan}
H.~Cartan and S.~Eilenberg, \emph{Homological algebra}, Princeton Landmarks in
  Mathematics, Princeton University Press, Princeton, NJ, 1999, With an
  appendix by David A. Buchsbaum, Reprint of the 1956 original.

\bibitem{Chen}
X.~Chen, S.~Yang, and G.~Zhou, \emph{{B}atalin-{V}ilkovisky algebras and the
  noncommutative {P}oincar\'e duality for {K}oszul {C}alabi-{Y}au algebras}, J.
  Pure Appl. Algebra \textbf{220} (2016), 2500--2532.

\bibitem{Connes}
A.~Connes, \emph{Non-commutative differential geometry}, Publ. Math. IH\'{E}S
  \textbf{62} (1985), 257--300.

\bibitem{GaoHou}
J.~Gao and B.~Hou, \emph{{B}atalin-{V}ilkovisky structure on {H}ochschild
  cohomology of self-injective quadratic monomial algebras}, Chinese Quarterly
  Journal of Mathematics \textbf{36} (2021), no.~3, 320--330.

\bibitem{Gerstenhaber}
M.~Gerstenhaber, \emph{The cohomology structure of an associative ring}, Ann.
  of Math. \textbf{78} (1963), 267--288.

\bibitem{Ginzburg}
V.~Ginzburg, \emph{{C}alabi-{Y}au algebras}, arXiv:math/0612139 (2006).

\bibitem{Hochschild}
G.~Hochschild, \emph{On the cohomology groups of an associative algebra}, Ann.
  Math. (2) \textbf{46} (1945), 58--67.

\bibitem{Itagaki}
T.~Itagaki, \emph{{B}atalin-{V}ilkovisky algebra structures on the {H}ochschild
  cohomology of self-injective {N}akayama algebras}, Journal of Mathematics
  \textbf{59} (2023), no.~1, 33--59.

\bibitem{Kowalzig}
N.~Kowalzig and U.~Kraehmer, \emph{{B}atalin-{V}ilkovisky {S}tructures on {E}xt
  and {T}or}, Journal für die reine und angewandte Mathematik (Crelles
  Journal) (2014), no.~697, 159--219.

\bibitem{Lambre}
T.~Lambre, G.~Zhou, and A.~Zimmermann, \emph{The {H}ochschild cohomology ring
  of a {F}robenius algebra with semisimple {N}akayama automorphism is a
  {B}atalin-{V}ilkovisky algebra}, J. Algebra \textbf{446} (2016), no.~15,
  103--131.

\bibitem{LiuMa}
L.~Liu and W.~Ma, \emph{{B}atalin-{V}ilkovisky algebra structures on
  {H}ochschild cohomology of generalized {W}eyl algebras}, Frontiers of
  Mathematics \textbf{17} (2021), 915--941.

\bibitem{LiuZhou}
Y.~Liu and G.~Zhou, \emph{{T}he {B}atalin-{V}ilkovisky structure over the
  {H}ochschild cohomology ring of a group algebra}, J. Noncommut. Geom.
  \textbf{10} (2016), 811--858.

\bibitem{Lu}
W.~L\"u, \emph{{T}he {B}atalin-{V}ilkovisky structure over the {H}ochschild
  cohomology ring of exterior algebras}, Frontiers of Mathematics \textbf{18}
  (2023), 903--934.

\bibitem{Menichi}
L.~Menichi, \emph{{B}atalin-{V}ilkovisky algebra structure on {H}ochschild
  cohomology}, Bulletin de la SMF \textbf{137} (2009), no.~2, 277--295.

\bibitem{SuarezAlvarez}
M.~Su\'arez-\'Alvarez, \emph{{T}he action of the {N}akayama automorphism of a
  {F}robenius algebra on {H}ochschild cohomology}, arXiv:2502.04546 (2025).

\bibitem{Tamarkin}
D.~Tamarkin and B.~Tsygan, \emph{Noncommutative differential calculus, homotopy
  {BV} algebras and formality conjectures}, Methods Funct. Anal. Topology
  \textbf{6} (2000), no.~2, 85--100.

\bibitem{Tradler}
T.~Tradler, \emph{The {B}atalin-{V}ilkovisky algebra on {H}ochschild cohomology
  induced by infinity inner products}, Ann. Inst. Fourier \textbf{58} (2008),
  no.~7, 2351--2379.

\bibitem{Volkov}
Y.~Volkov, \emph{{BV}-differential on {H}ochschild cohomology of {F}robenius
  algebras}, J. Pure. Appl. Algebra \textbf{220} (2016), no.~10, 3384--3402.

\bibitem{Witherspoon}
S.J. Witherspoon, \emph{Hochschild cohomology for algebras}, Graduate Studies
  in Mathematics, vol. 204, American Mathematical Society (AMS), 2019.

\end{thebibliography}

\end{document}